\documentclass[12pt]{amsart}
\usepackage{amsmath}
\usepackage{graphicx}
\usepackage{hyperref}
\usepackage[latin1]{inputenc}
\usepackage{mathtools}
\usepackage{amsfonts}
\usepackage{amssymb}
\usepackage[T1]{fontenc}
\usepackage{amsthm}
\usepackage{fullpage}
\usepackage{xcolor}

\theoremstyle{remark}

\theoremstyle{theorem}
\newtheorem{theorem}{Theorem}[section]

\newtheorem{lemma}[theorem]{Lemma}
\theoremstyle{definition}
\newtheorem{definition}{Definition}[section]
\newtheorem{conjecture}{Conjecture}[theorem]

\theoremstyle{remark}

\numberwithin{equation}{section}

\newcommand{\F}{\mathbb{F}_q}
\newcommand{\Fm}{\mathbb{F}_{q^m}}
\newcommand{\M}{\mathbf{M}}

\newcommand{\Tr}{Tr_{\Fm/\F}}
\newcommand{\T}{\mathfrak{T}}
\newcommand{\rit}{2\times W(r)^2 \times \Lambda}

\title{Pair of primitive elements in quadratic form with prescribed trace over a finite field }

\author{ Himangshu Hazarika}
\address{Department of Mathematical Sciences, Tezpur University, Assam, India}
\email{diku\_95@tezu.ernet.in}

\author{Dhiren Kumar Basnet}
\address{Department of Mathematical Sciences, Tezpur University, Assam, India} 
\email{dbasnet@tezu.ernet.in}

\thanks{This work was funded by the Council of Scientific and Industrial Research, New Delhi, Government of India's research grant no.~09/796(0099)/2019-EMR-I}

\begin{document}
\maketitle

\begin{abstract}
In this article, we establish a sufficient condition for the existence of primitive element $\alpha\in \Fm$ is such that $f(\alpha)$ is also primitive element of $\Fm$ and $Tr_{\Fm/\F}(\alpha)=\beta$, for any prescribed $\beta\in\F$, where $f(x)= ax^2 + bx + c\in \Fm(x)$ such that $b^2-4ac\neq 0$. We conclude that, for $m\geq 5$ there is only one exceptional pair $(q,m)$ which is $(2,6)$.

\end{abstract}



\section{Introduction}
Let $q$ be a prime power and we denote by $\F$, the finite field of order $q$ and its extension field by $\Fm$. A {\em primitive} element of $\Fm$ is a generator of the (cyclic)  multiplicative group $\mathbb{F}^*_{q^m}$. There are $\phi(q^m-1)$ primitive elements in the field $\Fm$, where $\phi$ is the Euler-phi function. For an element $\alpha\in \Fm$, the \emph{trace} of the element $\alpha$ in $\Fm$ over $\F$ is defined as  $Tr_{\Fm/\F}(\alpha)= \alpha +\alpha^q+...+ \alpha^{q^{m-1}}$. 

Cohen \cite{SDC} proved that for every $\beta\in \F$, the extension field $\Fm$ contains a primitive element $\alpha$ such that $Tr_{\Fm/\F}(\alpha)=\beta$, if $m\geq 3$ and $(q,m)\neq (4,3)$. Further, if $m=2$ or $(q,m)= (4,3)$, for every element $\beta$ in $\F^*$, there exists a primitive element $\alpha$ of $\Fm$, such that $Tr_{\Fm/\F}(\alpha)= \beta$. Cao and Wang \cite{XP} proved the existence of a primitive element $\alpha\in \Fm$ such that   $\alpha +\alpha^{-1}$ is also primitive and $Tr_{\Fm/\F}(\alpha)= a$ and $Tr_{\Fm/\F}(\alpha^{-1})= b$, for any $a,b\in \F^*$, for all $q$ and $m\geq 29$. 

In 2018, Anju Gupta, R K Sharma and S D Cohen \cite{ARS} established the following results on the primitive pair $(\alpha, \alpha+\alpha^{-1})$.


\begin{theorem}
For $q$, a positive prime power and $m\geq 5$ an integer, $\Fm$ contains a primitive element $\alpha$ such that $ \alpha+\alpha^{-1}$ is also primitive element of $\Fm$ with $Tr_{\Fm/\F}(\alpha)=\beta$, for any prescribed $\beta$ in $\F$, unless $(q,m)$ is one of the following pairs, $(2,5), (2,6), (2,8), (2,9), (2,10),$\\ 
$(2,12), (3,5), (3,6), (3,7), (3,8), (3,9), (3,12), (4,5), (4,6), (4,7), (4,8), (5,5), (5,6), (5,8), (7,5),$\\
$ (7,6), (7,7), (8,5), (8,6), (8,8), (9,5), (9,6), (11,5), (11,6), (13,5), (13,6), (16,5), (16,6), (17,6),$\\
$ (19,5), (19,6), (23,6), (25,5), (25,6), (29,6), (31,5), (31,6), (37,5), (43,5), (49,5), (61,5), (61,6) $ and $(71,5)$.

\end{theorem} 

\begin{theorem}
Let $\F$ be finite field of order $q$.  Whenever $m\geq 5$, $\Fm$ contains a primitive element $\alpha$ such that $\alpha+\alpha^{-1}$ is also primitive and $Tr_{\Fm/\F}(\alpha)=\beta$, for every $\beta\in \F$.
\end{theorem}

In this paper we establish the existence of a primitive pair $(\alpha, a\alpha^2+b\alpha+c)$ i.e. both $\alpha$ and $a\alpha^2+b\alpha+c$ are primitive elements of  $\Fm$ (where $b^2\neq 4ac$)  with $Tr_{(\Fm|\F)}(\alpha)= \beta$, for any prescribed $\beta \in \F$. We prove the following results. 

\begin{theorem}
Suppose $q$ be some odd prime power and $m\geq 5$ is a positive integer. Then there exists a primitive element $\alpha\in \Fm$ such that $a\alpha^2+b\alpha+c$ is also primitive in $\Fm$ with $Tr_{\Fm/\F}(\alpha)=\beta$ for any prescribed $\beta\in \F$, unless one of the following holds.
\begin{itemize}
\item  $q=3$ and $5\leq m \leq 8$ or $m=12$;
\item  $q=5$ and $m= 5,6,8$;
\item  $q=7$ and $m= 5,6,7$;
\item  $q=9,11,13$ and $m=5,6$;
\item for $q\geq 17$, $(q,m)$ is one of the pairs $(17,6), (19,5), (19,6), (23, 6), (25,5), (29,6), 
(31,5),\\ (31,6)$ and $(37,5)$.
\end{itemize}
\end{theorem}

\begin{theorem}
Suppose $\Fm$ be field of even characteristic and $m\geq 5$ is a positive integer. Then there exists a primitive element $\alpha\in \Fm$ such that $a\alpha^2+b\alpha+c$ is also primitive in $\Fm$ with $Tr_{\Fm/\F}(\alpha)=\beta$ for any prescribed $\beta\in \F$, unless one of the following holds.
\begin{itemize}
\item  $q=2$ and $m=6$ or $8\leq m \leq 12$;
\item  $q=4$ and $5 \leq m \leq 8$;
\item  $q=8$ and $m= 5,6,8$;
\item for $q\geq 16$, $(q,m)$ is one of $(16,5), (16,6)$.
\end{itemize}
\end{theorem}

After computation in SAGE Math we have the following theorem.

\begin{theorem}
Suppose $q$ be some prime power and $m\geq 5$ a positive integer. Then for every $\beta\in \F$, there exists a primitive element $\alpha\in \Fm$ such that $a\alpha^2+b\alpha+c$ is also primitive in $\Fm$ with $Tr_{\Fm/\F}(\alpha)=\beta$, unless $(q,m)$ is $(2,6)$.
\end{theorem}

In this article we don't  consider the case $m=1$. As in that case $Tr_{\Fm/\F}(\alpha)= \alpha$. To satisfy the condition, for every primitive element $\alpha\in \F$,  $a\alpha^2+b\alpha+c$ in $\F$ must be primitive element, which is  possible only if $q-1$ is Mersenne prime, i.e. $q$ must be even prime power. Since the case is trivial, hence we don't discuss it further. Similarly when $m=2$, then for $0\in\F$, we must have some primitive element $\alpha\in\Fm$ such that $\alpha+\alpha^q= 0$, which is a contradiction. 
Thus we consider that $m\geq 3$. When we studied case $m=2,3$, we had sufficiently large numbers of possible exceptional pairs, which inspired us to treat the mentioned cases in details in another future article. Hence in this article we discuss for $m\geq 5$.

\section{Preliminaries}
 Now for $\alpha\in \mathbb{F}^*_{q^m}$, the multiplicative order is denoted by ord($\alpha$) and $\alpha$ is primitive if and only if ord$(\alpha)=q^m-1$. From the definition, it is clear that $q^m-1$ can be freely replaced by the radical $m_0$, where $m_0$ is such that $m=m_0p^a$, where $a$ is a non negative integer and gcd$(m_0,p)=1$.

  In this article, we use the following definitions and lemmas in our result.
 \begin{definition}
 For a finite abelian group $G$, a character $\chi$ of $G$ is a homomorphism from $G$ into the group $C:= \{z \in \mathbb{C} : |z| = 1\}$.  A $\mathit{dual\thinspace group}$ or $\mathit{character\thinspace group}$ of $G$ is a group of characters of $G$ under multiplication,  which is denoted by $\widehat{G}$. We observe that $\widehat{G}$ is isomorphic to $G$. Also, the trivial character of $G$ is denoted by $\chi_0$ and is  defined as $\chi_0(a)=1$ for all $a \in G$.
  \end{definition}

  In $\Fm$, there are two types of abelian groups, the additive group $\Fm$ and the multiplicative group $\mathbb{F}^*_{q^m}$. Hence, there are two types of characters of $\Fm$, which are $\mathit{additive \thinspace characters }$ of $\Fm$ and $\mathit{multiplicative \thinspace characters }$ of $\mathbb{F}^*_{q^m}$. Multiplicative characters can be extended from $\mathbb{F}^*_{q^m}$ to $\Fm$ as the following
 \hspace{.1cm}  $\chi(0)=\begin{cases}
                         0 \,\mbox{ if}\, \chi\neq\chi_0,\\
                         1 \,\mbox{ if}\, \chi=\chi_0.
                         \end{cases} $

   Since $\widehat{\mathbb{F}^*_{q^m}} \cong \mathbb{F}^*_{q^m}$, so $\widehat{\mathbb{F}^*_{q^m}}$ is cyclic and for any divisor $d$ of $q^m-1$, there are exactly $\phi(d)$ characters of order $d$ in $\widehat{\mathbb{F}^*_{q^m}}$.
   
      For $r|q^m-1$, an element $\alpha$ in $\Fm$ is called $r$-free if $d|r$ and $\alpha= \gamma^d$, for some $\gamma\in \Fm$ implies $d=1$. From the definition, it is clear that an element $\alpha$ is primitive if and only if $\alpha$ is $(q^m-1)$-free.

 For any $r|q^m-1$, the characteristic function for the subset of $r$-free elements of $\mathbb{F}^*_{q^m}$ defined by Cohen and Huczynska  \cite{CH1, CH2} is as follows
 $$\sigma_r: \alpha\mapsto\theta(r)\underset{d|r}{\sum}\left(\frac{\mu(d)}{\phi(d)}\underset{\chi_d}{\sum}\chi_d(\alpha)\right),$$
     where $\theta(r):=\frac{\phi(r)}{r}$, $\mu$ is the M\"obius function and $\chi_d$ stands for any multiplicative character of order $d$.
   For any $r|q^m-1$, we use ``integral'' notation for weighted sums as follows
\begin{align*}
\underset{d|q^m-1}{\int}\chi_d:= \underset{d|q^m-1}{\sum}\frac{\mu(d)}{\phi(d)}\underset{\chi_d}{\sum}\chi_d.
\end{align*}

 Thus the  characteristic function for the subset of $r$-free elements of $\Fm^*$ becomes
\begin{align*}
\sigma_r: \alpha\mapsto\theta(r)\underset{d|r}{\int}\,\chi_d(\alpha).
\end{align*}

Following the above,  the characteristic function of the set of elements $\alpha \in \Fm$  with  $\Tr(\alpha)= \beta\in \F$ is given by

\begin{equation*}
\Gamma_\beta: \alpha \mapsto \frac{1}{q} \underset{\psi \in \widehat{\Fm}}{\sum} \psi (\Tr(\alpha)- \beta).
\end{equation*}

in which the sums are over all additive characters $\psi$ of $\Fm$, i.e., all elements of $\widehat{\Fm}$.

Since every additive character $\psi$ of $\F$ can be obtained by  the canonical additive character  $\psi_0$ of $\F$  as $\psi(\alpha) = \psi_0(u\alpha)$, where $u$ is some element of $\F$.

\begin{eqnarray}
\Gamma_\beta(\alpha) &=  \frac{1}{q} \underset{u \in \F}{\sum} \psi_0 (\Tr(u\alpha)- u\beta)\nonumber \\
&= \frac{1}{q} \underset{u \in \F}{\sum} \widehat{\psi_0}(u\alpha)\psi_0(u\beta),  
\end{eqnarray}
where $\widehat{\psi_0}$ is the additive character of $\Fm$ defined by $\widehat{\psi_0}(\alpha)= \psi_0(\Tr(\alpha)).$

\begin{lemma} {\bf\cite{RH} }
If $\chi$ is any nontrivial character of a finite abelian group $G$ and $\alpha\in G$ any nontrivial element, then
$$\underset{\alpha\in G}{\sum}\chi(\alpha)=0 \quad  \mbox{and} \quad \underset{\chi\in \widehat{G}}{\sum}\chi(\alpha)=0.$$
\end{lemma}.

We will use the following lemmas in our main results. 

\begin{lemma} \label{charbound1}
{\bf  (\cite{DW})}
Consider any two nontrivial multiplicative characters $\chi_1,\chi_2$ of the finite field $\mathbb{F}_{q^m}$. Again, let $f_1(x)$ and $f_2(x)$ be two monic  co-prime polynomials in $\mathbb{F}_{q^m}[x]$, such that at least one of $f_i(x)$ is not of the form $g(x)^{ord(\chi_i)}$ for $i=1,2$; where $g(x)\in \mathbb{F}_{q^m}[x]$. Then
$$\Big|\underset{\alpha\in \mathbb{F}_{q^m}}{\sum}\chi_1(f_1(\alpha))\chi_2(f_2(\alpha))\Big|\leq (n_1+n_2-1)q^{m/2},$$
 where $n_1$ and $n_2$ are the degrees of largest square free divisors of $f_1$ and $f_2$, respectively.
\end{lemma}
\begin{lemma}\label{charbound2} {\bf (\cite{LDQ}) } Let $f_1(x), f_2(x),\ldots, f_k(x)\in \mathbb{F}_{q^m}[x]$ be distinct irreducible polynomials over $\mathbb{F}_{q^m}$. Let $\chi_1,\chi_2, \ldots,\chi_k$ be multiplicative characters and $\psi$ be a non-trivial additive character of $\mathbb{F}_{q^m}$. Then
 $$\left | \underset{\underset{f_i(\alpha)\neq0}{\alpha\in\mathbb{F}_{q^m}}}{\sum} \chi_1(f_1(\alpha))\chi_2(f_2(\alpha))\ldots\chi_k(f_k(\alpha))\psi(\alpha)\right|\leq n\,q^{m/2},$$
 where $n= \overset{k}{\underset{j=1}{\sum}}deg(f_j)$.
\end{lemma}

\section{Existence of primitive pairs with prescribed trace}

We  establish a sufficient condition for every $\beta\in\F$, the existence of a primitive pair $(\alpha, f(\alpha))\in\Fm$ such that $Tr_{\Fm/\F}(\alpha)=\beta$. Take $r_1,r_2$ such that $r_1, r_2|q^m-1$.
Define $\M(r_1,r_2)$ to be the number of $\alpha\in\Fm$ such that $\alpha$ is $r_1$-free and $f(\alpha)$ is $r_2$-free, where $f(x)=ax^2+bx+c$ and $a,b,c\in\mathbb{F}_{q^m}$, $b^2-4ac\neq 0$ and $\Tr(\alpha)= \beta \in \F$. Hence it is sufficient to show that $\M(q^m-1,q^m-1)>0$ for every $\beta\in \F$.

 We use the notation $\omega(n)$ to denote the number of prime divisors of $n$. For calculations we use $W(n):=2^{\omega(n)}$.

 \begin{theorem}
 Let $\beta\in \F$ and $r_1,r_2| q^m-1$. Then $\M(r_1,r_2)> 0$ if $q^{m/2-1}> 2W(r_1)W(r_2)$.

In particular if $q^{m/2-1}> 2W(q^m-1)^2$,  then $M(q^m-1,q^m-1)>0$.

 \end{theorem}

\begin{proof}
By definition $\M(r_1, r_2) = \underset{\alpha\in \Fm^*}{\sum}\sigma_{r_1}(\alpha)\sigma_{r_2}(f(\alpha))\Gamma(\alpha)$. 
Now
\begin{equation}
\M(r_1, r_2)= \frac{\theta(r_1)\theta(r_2)}{q}\underset{\underset{d_1|r_1}{d_2|r_2}}{\sum} \frac{\mu(d_1)\mu(d_2)}{\phi(d_1)\phi(d_2)}\underset{\chi_{d_1},\chi_{d_2}}{\sum} S(\chi_{d_1}, \chi_{d_2})  ,
\end{equation}

where \begin{equation*}
 S(\chi_{d_1}, \chi_{d_2})= \underset{u \in \F}{\sum}\psi_0(-\beta u)\underset{\alpha\in \Fm^*}{\sum}\chi_{d_1}(\alpha)\chi_{d_2}(a\alpha^2+b\alpha+c)\widehat{\psi_0}(u\alpha).
\end{equation*}

Now, we have  $\chi_{d_1}(x)= \chi_{q^m-1}(x^{m_i})= \chi(x^{m_i})$ for $0\leq m_i \leq q^m-2$.

\begin{eqnarray}
S(\chi_{d_1}, \chi_{d_2}) &= \underset{u \in \F}{\sum}\psi_0(-\beta u)\underset{\alpha\in \Fm^*}{\sum}\chi(\alpha^{m_1}(a\alpha^2+b\alpha+c)^{m_2})\widehat{\psi_0}(u\alpha) \nonumber \\
&= \underset{u \in \F}{\sum}\psi_0(-\beta u)\underset{\alpha\in \Fm^*}{\sum}\chi(F(\alpha))\widehat{\psi_0}(u\alpha), \nonumber 
\end{eqnarray}

where $F(x)= x^{m_1}(ax^2+bx+c)^{m_2}\in \Fm[x]$, $0\leq m_1, m_2 \leq q^m-2$.


If $F(x)\neq y H^{q^m-1}$ for any $y\in \Fm$ and $H\in \Fm[x]$, then by following the Lemma 2.2 \cite{DW} we have 

\begin{equation*}
\left| \underset{\alpha \in \Fm^*}{\sum}\chi(F(\alpha))\psi_0(u\alpha) \right| \leq (0+3+1-0-2)q^{m/2}= 2q^{m/2}. 
\end{equation*}
Then we have 

\begin{eqnarray}
\left|S\right| & = \left| \underset{u \in \F}{\sum}\psi_0(-\beta u)\underset{\alpha\in \Fm^*}{\sum}\chi(F(\alpha))\widehat{\psi_0}(u\alpha)\right| \nonumber \\
& \leq \underset{u \in \F}{\sum} \left| \psi_0(-\beta u)\underset{\alpha\in \Fm^*}{\sum}\chi(F(\alpha))\widehat{\psi_0}(u\alpha)\right| \nonumber \\
& = \underset{u \in \F}{\sum}\left| \psi_0(-\beta u)\right| \left|\underset{\alpha\in \Fm^*}{\sum}\chi(F(\alpha))\widehat{\psi_0}(u\alpha)\right| \nonumber\\
& \leq  \underset{u \in \F}{\sum} \left| \psi_0(-\beta u)\right| 2q^{m/2}= 2q^{m/2+1}.\nonumber
\end{eqnarray}

Let $F= yH^{q^m-1}$ for some $y \in \Fm$ and $H\in \Fm[x]$. 

Then we have $x^{m_1}(ax^2+bx+c)^{m_2}= yH^{q^m-1}$.

If $m_1\neq 0$, then $(ax^2+bx+c)^{m_2}= y x^{q^m-1-m_1}B^{q^m-1}$, where $B(x)= \frac{H(x)}{x}$. Comparing the degrees we have 
\begin{equation}\label{deg}
q^m-1-m_1+k(q^m-1) \Rightarrow m_1+2m_2= (k+1)(q^m-1),
\end{equation} 
where $k$ is the degree of $B$.

Since $m_1\neq q^m-1$, hence we have $c=0$. Then


\begin{eqnarray}\label{eq2}
(ax^2+bx)^{m_2} &= y x^{q^m-1-m_1}B^{q^m-1}\nonumber \\
\Rightarrow x^{m_2}(ax+b)^{m_2} &= y x^{q^m-1-m_1}B^{q^m-1}.
\end{eqnarray}

If $m_2\neq 0$ then we have 

\begin{equation}\label{eq3}
(ax+b)^{m_2}= y x^{q^m-1-m_1-m_2}B^{q^m-1}.
\end{equation}

If $q^m-1 \neq m_1+ m_2$, then $b=0$, which is a contradiction. Hence $q^m-1= m_1+ m_2$. From \ref{deg}, we have $m_2+ (q^m-1) = (k+1)(q^m-1)$ i.e., $m_2= k(q^m-1)$, which is possible only if $k=0$.  Again form \ref{eq2}, we have $a=0$, which is a contradiction. Hence $m_2$ must be zero and thus $m_1$ must be zero or $m_1=m_2=0$.

Thus, when $F= yH^{q^m-1}$, $(\chi_{d_1}, \chi_{d_2})= (\chi_1, \chi_2)$. 

 If $u\neq 0$, then $\left| S(\chi_{d_1}, \chi_{d_2}) \right| = q-1 \leq 2q^{m/2+1}$.

Hence $\left| S(\chi_{d_1}, \chi_{d_2}) \right| \leq 2q^{m/2+1}$, when $(\chi_{d_1}, \chi_{d_2}, u) \neq (\chi_1, \chi_1, 0)$. Thus we have

\begin{equation}
\M(r_1,r_2)\geq \frac{\theta(r_1)\theta(r_2)}{q} \left( q^m-1- 2q^{m/2+1}(W(r_1)W(r_2)-1)\right).
\end{equation}

Hence $\M(r_1,r_2)>0$ i.e., the sufficient condition is $q^{m/2-1}> 2W(r_1)W(r_2)$.

\end{proof}

In particular if \begin{equation}\label{geneq}
q^{m/2-1}> 2W(q^m-1)^2,
\end{equation}
then $M(q^m-1,q^m-1)>0$.

\section{Prime Sieve Technique}

\begin{lemma} Let $r|q^m-1$ and $p$ any prime dividing $q^m-1$ but not $r$.
\begin{equation*}
\left| \M(pr,r)- \M(r,r)\right| \leq \frac{2\theta(r)^2\theta(p)}{q}W(r)^2q^{m/2+1}
\end{equation*}
and 
 \begin{equation*}
\left| \M(r,pr)- \M(r,r)\right| \leq \frac{2\theta(r)^2\theta(p)}{q}W(r)^2q^{m/2+1}
\end{equation*}
\end{lemma}

This proof is identical with lemma 3.2 of \cite{ARS}, hence we omit the proof.


\subsection{ Sieving inequality}

\begin{lemma}\label{sieveineq}
{\bf(Sieve Inequality)}\cite{ARS}
Let $r$ be a divisor of $q^m-1$ and $p_1, p_2,\dots , p_n$ be the distinct primes dividing $q^m-1$ but not $r$.  Abbreviate  $ \mathfrak{M}(q^m-1, q^m-1) $ to $\mathfrak{M}$. Then
\begin{equation} \label{ineq}
 \mathfrak{M} \geq \overset{n}{\underset{i=1}{\sum}}\mathfrak{M}(p_i r, r)+  \overset{n}{\underset{i=1}{\sum}}\mathfrak{M}( r,p_i r)-(2n-1)\M(r, r).
\end{equation}
\end{lemma}

Applying the sieving inequality we have the following lemma.

\begin{theorem}\label{delta}
With all the assumptions as in Lemma $ \ref{sieveineq}$,  define
$ \Delta := 1 - 2  \overset{n}{\underset{i=1}{\sum}}\frac{1}{p_i} $
 and $ \Lambda := \frac{2n+k-1}{\Delta}+2.$
 Suppose  $\Delta>0$. Then a sufficient condition for the existence of a primitive element $\alpha$ for which $a\alpha^2+b\alpha+c$ is also primitive in $\mathbb{F}_{q^m}$ with $b^2-4ac\neq 0$ and $\Tr(\alpha)= \beta$  is

\begin{equation}\label{cond3}
  q^{m/2-1}> 2 {W(d)}^2\Lambda .
  \end{equation}
\end{theorem}

\begin{proof}   A key step is to write (\ref{ineq}) in the equivalent form
\begin{equation}\label{altineq}
 \mathfrak{M}  \geq \frac{\theta(r)^2}{q}\left( \sum_{i=1}^n 2\theta(p_i)\left(  -2q^{m/2+1}W(r)^2\right)  +  \Delta \left( q^m-1-2q^{m/2+1}(W(r)^2-1)\right)   \right).
\end{equation}


\begin{equation*}
\mathfrak{M}\geq \frac{\theta(r)^2}{q}\Delta\left[ \left(\frac{2\sum_{i=1}^n\theta(p_i)}{\Delta}+1\right) \{-2q^{m/2+1}W(r)^2\}+ \{ q^m-1+2q^{m/2+1}\}\right].
\end{equation*}

Since $\Delta= 2\sum_{i=1}^n\theta(p_i) - (2n-1)$ and  $\Lambda= \left(\frac{2\sum_{i=1}^n\theta(p_i)}{\Delta}+1\right)$, we have 

\begin{equation*}
\mathfrak{M}\geq \frac{\theta(r)^2}{q}\Delta\left[ -2\Lambda q^{m/2+1}W(r)^2+q^m-1+2q^{m/2+1}\right].
\end{equation*}

We already have $\Delta>0$, hence $\M>0$ if $q^{m/2-1}> q^{-m/2-1}-2+2W(r)^2\Lambda$ i.e., if \\ $q^{m/2-1}>2W(r)^2\Delta$. That is it is sufficient if  $q^{m/2-1}> 2W(r)^2\Lambda$, giving us our sufficient condition.
 \end{proof}

\section{Some estimates for existence of primitive pair with prescribed trace}
For the remaining part of the article we are going to denote by $\T$,  the set of all pairs $(q,m)$ for which $\mathbb{F}_{q^m}$ contains primitive pair $(\alpha,f(\alpha))$ with prescribed trace.
First we have the following lemma.
\begin{lemma}\cite{SDC1}
For any $n,\alpha\in \mathbb{N}$, $W(n)\leq B_{\alpha ,n} n^{1/\alpha}$, where $B_{\alpha,n}=\frac{2^\alpha}{(p_1p_2\ldots p_k)^{1/\alpha}}$ and $p_1, p_2,\ldots, p_k$ are primes $\leq 2^\alpha$ that divide $n$ and $W$ denotes the same as described before. 
\end{lemma}

From this, we have the following lemma:

\begin{lemma}\label{Wbound}
With the same meaning of $W$ as mentioned above, we have
\begin{itemize}
\item For $\alpha=6$, $W(n)< 37.4683 n^{1/6}$,
\item For $\alpha=8$, $W(n)< 4514.7 n^{1/8}$,
\item For $\alpha=14$, $W(n)< (5.09811\times 10^{67}) n^{1/14}$.
\end{itemize}
\end{lemma}

\begin{lemma}\label{L5.2}
Let $N$ be a positive integer such that $\omega(N)\geq 85$, i.e. $N> 2.41\times 10^{179}$. Then $W(N)< N^{1/7}$.
\end{lemma}
\begin{proof}
The product of first 85 primes exceeds $M=2.41\times 10^{179}$. We write $N=n_1n_2$, a product of coprime integers, where all primes dividing $n_1$ are amongst the least 85 primes dividing $N$ and those dividing $n_2$ are larger primes. Hence $n_1>M$ and $n_1^{1/7}> M^{1/7}> 4.23 \times 10^{25} $, where as $W(n_1)=2^85< 1.93\times 10^{25}$. Since $\mathfrak{p}^{1/7}>2$ for all primes $\mathfrak{p}>435$ ($85^{th}$ prime), the result follows.
\end{proof}

\begin{theorem}\label{T5.3}
Let $q$ be some prime power, i.e. $q=p^s$, where $p$ is any prime and $s$ is a positive integer. Then $(q,m)\in \T$ when $q\geq 16$ and $m\geq 29$. 
\end{theorem}
\begin{proof}
From the Lemma \ref{Wbound} we have $W(n)<4514.7 n^{1/8}$. Using this in the condition $q^{m/2-1}> 2W(q^m-1)^2$, we have a sufficient condition  as $q^{m/4-1}>2.03825 \times 10^7$. When $q\geq 16$, then the condition holds for $m\geq 29$. 
\end{proof}

For the remaining cases i.e. $q=2, 3, 4, 5, 7, 8, 9, 11, 13$, the condition holds for $m= 102, 64, 53, 46, 39, 37, 35, 33, 31$ respectively. We check the pairs by further calculation.

\subsection{Estimates for fields of odd characteristic}

Let $q$ be a odd prime power. As already discussed in Theorem \ref{T5.3}, $ m\geq 29$ for $q\geq 16$, in this section we assume that $m\leq 28$. We use the notation $\omega$ to denote $\omega(q^m-1)$ throughout this article. We begin with the following theorem.

\begin{theorem}
Suppose $\Fm$ be a field of odd characteristic and $m\geq 5$ is a positive integer. Then there exists a primitive element $\alpha\in \Fm$ such that $a\alpha^2+b\alpha+c$ is also primitive in $\Fm$ with $Tr_{\Fm/\F}(\alpha)=\beta$ for any prescribed $\beta\in \F$, unless $(q,m)$ is one of the pairs $(3,5), (3,6), (3,7), (3,8), (3,12), (5,5), (5,6), (5,8), (7,5), (7,6), (7,7), (9,5), (9,6), (11,5),$\\
$ (11,6), (13,5), (13,6), (17,6), (19,5), (19,6), (23, 6), (25,5), (29,6), 
(31,5), (31,6)$ and $(37,5)$.
\end{theorem}   

\begin{proof}
At first we assume that $\omega\geq 85$. Then  from the sufficient condition \ref{geneq}, applying the Lemma \ref{L5.2}, we have the inequality as $q^{\frac{3m}{14}-1}>2$. Which holds unless $q\leq 16411$ if $m=5$ or $q\leq 13$ if $m=6$, this implies that $\omega< 85$.

Next we assume that $21\leq \omega \leq 84 $. Then as in Theorem \ref{sieveineq}, let $r$ be the divisor of $q^m-1$, whose prime factors are atleast 21 primes dividing  $q^m-1$,  and hence $r\leq63$. In Theorem \ref{delta} when $r=63$, we have $\Delta> 0.373307$ and hence $\Lambda < 336.845.$ Thus $ \rit < 2.96292 \times 10^{15}$. Now the sufficient condition is  $q^m > (\rit)^{\frac{2m}{m-2}}$, which becomes $q^m > (\rit)^{\frac{10}{3}}$ (as $m\geq 5$). Hence the inequality becomes $q^m> 3.73594 \times 10^{51}$. In fact when $\omega \geq 33$, then $q^m-1$ is the product of atleast first 33 primes, i.e. $q^m> 7.2\times 10^{52}$. So $(q,m) \in \T$ when $\omega \geq 33$ or $q> (3.73\times 10^{51})^{1/m}$. When $m=5$, $q> 2.06291 \times 10^{10}$.

Then we assume $4\leq  \omega \leq 32$ and $q < 2.06291 \times 10^{10}$. Then proceeding as above we assume $\omega(r) = 6$, $r=26$ then $\Delta > 0.0626089$ and hence $\Lambda < 816.581$. Hence $\rit < 6.68943 \times 10^{22}$, i.e. the sufficient condition is  $q^m> 5.60424 \times 10^{22}$. When $\omega \geq 18$,  then $q^m> 1.17288 \times 10^{23}$. We conclude that $(q,m) \in \T$ when $\omega \geq 18$ or $q> (5.64024 \times 10^{22})^{1/m}$. Finally $q> 35503$ when $m=5$.

In the next stage our range for $\omega$ is $4\leq  \omega \leq 17$ and  $q< 35503$. We take $\omega (r)= 5$ and then $r=12$, then $\Delta > 0.139272$ and hence $\Lambda < 167.144$. Hence $\rit < 342311$ i.e. $q^m > 2.80588 \times 10^{18}$. In fact when $\omega \geq 16$, then $q^m> 3.25891 \times 10^{19}$. From which we have $(q,m)\in \T$, when $\omega\geq 16$ or $q>(2.80588 \times 10^{18})^{1/m} $. Finally $q> 4894$ when $m=5$.

Then we assume $4\leq  \omega \leq 15$ and $q < 4894$. Then proceeding as above we take $\omega(r) = 5$, $r=10$, so $\Delta > 0.210906$ and hence $\Lambda < 92.0875$. Hence $\rit < 188595$, i.e. the sufficient condition is  $q^m> 3.84682 \times 10^{17}$. When $\omega \geq 15$,  then $q^m> 6.14889 \times 10^{17}$. We conclude that $(q,m) \in \T$, when $\omega \geq 15$ or $q> (3.84682 \times 10^{17})^{1/m}$. Finally $q> 3289$ when $m=5$.

Next we take $4\leq \omega \leq 14$, then repeating the above process with $\omega(r)= 4$, $r=10$, we have $\Delta > 0.0716411$ and $\Lambda < 267. 211$. Hence the inequality is $\rit < 136812$. The condition holds when $q > (1.31951 \times 10^{17})^{1/m}$. Then the condition holds when $m=5$, $q> 2655$, when $m=6$, $q> 714$; when $m=7$, $q> 279$; when $m=8$, $q> 138$; when $m=9$, $q> 80$; when $m=10$, $q> 52$; when $m=11$, $q> 36$; when $m=12$, $q> 27$; when $m=13$, $q> 21$; when $m=14$, $q> 17$; when $m=15$, $q> 13$ and when $m=16$, $q> 11$.

By calculation we have, if $\omega \leq 3$, then the pairs satisfy  \ref{geneq} with $r= q^m-1$.

Applying the Prime Sieve Technique with appropriate choice of $r$ from the factors of $q^m-1$ in \ref{sieveineq} as some of the calculations shown in Table 1, for $m\geq 5$, $(q,m)\in \T$ except possibly one of the pairs (3, 5), (3, 6), (3, 7), (3, 8), (3, 12), (5, 5), (5, 6), (5, 8), (7, 5), (7, 6), (7, 7), (9, 5), (9, 6), (11, 5), (11, 6), (13, 5), (13, 6), (17, 6), (19, 5), (19, 6), (23, 6), (25, 6), (29, 6), (31, 5), (31, 6) and (37, 7).

\end{proof}

Following table illustrates some of the calculations to show that pairs $(q, m)$ satisfying the condition $q^{m/2-1}> \rit$.

\vspace{.4cm}

\begin{tabular}{|c|c|c|c|c|c|c|}
\hline 
$(q,m)$ & primes of $q^m-1$ & $\omega(d)$ & $\Delta$ & $\Lambda$ & $q^{m/2-1}$ & $\rit$ \\ 
\hline 
(17, 5) & 2, 88741 & 1 & 0.999977 & 3.00002 & 70.0928 & 24.00002 \\ 
\hline 
(23, 5) & 2, 11, 292561 & 1 & 0.818175 & 5.6667 & 110.304 & 45.3336 \\ 
\hline 
(27, 5) & 2, 11, 13, 4561 & 1 & 0.663897 & 9.53129 & 140.296 & 76.2503 \\ 
\hline 
(29, 5)& 2, 7, 732541 & 1 & 0.714283 & 6.20002 & 156.17 & 49.6 \\ 
\hline 
(41, 5) & 2, 5, 579281 & 1 & 0.599997 & 7.00003 & 262.528 & 56.00002  \\ 
\hline 
(43, 5) & 2, 3, 7, 3500201 & 2 & 0.714285 & 6.2 & 281.97 & 198.4 \\ 
\hline 
(47, 5) & 2, 11, 23, 31, 14621 & 1 & 0.666572 & 12.5015 & 322.216 & 100.012 \\ 
\hline 
(61, 5) & 2, 3, 5, 131, 21491 & 2 & 0.58464 & 10.5523 & 476.425 & 337.674 \\ 
\hline  
(25, 6) & 2, 3, 7, 13, 31, 601 & 2 & 0.492561 & 16.2104 & 625 & 518.733 \\ 
\hline 
(27, 6) & 2, 7, 13, 19, 37, 757 & 1 & 0.39848 & 24.6426 & 729 & 197.141 \\ 
\hline 
(37, 6) & 2, 3, 7, 19, 31, 43, 67 & 2 & 0.468144 & 21.2249 & 1369 & 679.197 \\ 
\hline 
(41, 6) & 2, 3, 5, 7, 547, 1723 & 2 & 0.309469 & 24.6194 & 1681 & 787.821 \\ 
\hline 
(43, 6) & 2, 3, 7, 11, 13, 139, 631 & 2 & 0.361063 & 26.9264 & 1849 & 861.645 \\ 
\hline 
(61, 6) & 2, 3, 5, 7, 13, 31, 97, 523& 3 & 0.471481 & 21.08888 & 3721 & 2699.37 \\ 
\hline  
(5, 7) & 2, 19531 & 1 & 0.999898 & 3.0001 & 55.9017 & 24 \\ 
\hline 
(9, 7) & 2, 547, 1093 & 1 & 0.994514 & 5.01655 & 243 & 40.1324 \\ 
\hline 
(11, 7) & 2, 5, 43, 45319 & 1 & 0.553444 & 11.0343 & 401.312 & 88.2744 \\ 
\hline 
(19, 7) & 2, 3, 701, 70841 & 2 & 0.997119 & 5.00867 & 1573.56 & 160.277 \\ 
\hline 
(7, 8) & 2, 3, 5, 1201 & 2 & 0.598335 & 7.01391 &  343 & 224.445 \\ 
\hline 
(9, 8) & 2, 5, 17, 41, 193 & 1 & 0.42321 & 18.5403 & 729 & 148.322 \\ 
\hline 
(11, 8) & 2, 3, 5, 61, 7321 & 2 & 0.56694 & 10.8193 & 1331 & 346.218 \\ 
\hline  
(3, 9) & 2, 13, 757 & 1 & 0.843512 & 5.55656 & 46.7654 & 44.4525 \\ 
\hline 
(5, 9) & 2, 19, 31, 829 & 1 & 0.827807 & 8.04005 & 279.508 & 64.3204 \\ 
\hline 
(3, 10) & 2, 11, 61 & 1 & 0.785395 & 5.81973 & 81 & 46.5578 \\ 
\hline 
(5, 10) & 2, 3, 11, 71, 521 & 2 & 0.786174 & 8.35992 & 625 & 267.517 \\ 
\hline  
(3, 11) & 2, 23, 3851 & 1 & 0.912524 & 5.27858 & 140.296 & 42.3006 \\ 
\hline 
(5, 12) & 2, 3, 7, 13, 31, 601 & 2 & 0.492596 & 16.2104 & 3125 & 518.733 \\ 
\hline
(3, 13) & 2, 797161 & 1 & 0.999997 & 3 & 420.888 & 24 \\ 
\hline 
(3, 14) & 2, 547, 1093 & 1 & 0.994514 & 5.01655 & 729 & 40.1332 \\ 
\hline 
(3, 15) & 2, 11, 13, 4561 & 1 & 0.663897 & 9.53129 & 1262.67 & 76.2503 \\ 
\hline
\end{tabular} 

\begin{center}
Table 1
\end{center}

\subsection{Estimates for fields of even characteristic}

For the fields of even characteristic we use the following lemma.

\begin{lemma}\label{odd}\cite{SDC1}
For any odd positive integer $n$, $W(n) < 6.46 n^{1/5}$, where $W$ has the same meaning as mentioned above. 

\end{lemma}

\begin{theorem}
Let $\Fm$ be a field of even characteristic and $m\geq 5$. Then there exists a primitive pair $(\alpha, f(\alpha))$ ( where $f(x)= ax^2+ bx+ c $, with $b^2\neq 4ac$), with $Tr_{\Fm/ \F}(\alpha)= \beta$, for any prescribed $\beta\in \F$, unless $(q, m)$ is one of the pairs $(2, 6), (2, 8), (2, 9), (2, 10), (2, 11),$ $ (2, 12), (4, 5), (4, 6), (4, 7), (4, 8), (8, 5), (8, 6), (8, 8), (16, 5)$ and $(16, 6)$.
\end{theorem}

\begin{proof}
We will treat Mersenne primes differently. Here we will use the concept of radical of $m$ i.e. $m^\prime$, where $m^\prime$ is such that $m= 2^k m^\prime$ and $\gcd(2, m^\prime)=1$.

Applying Lemma \ref{odd} on the inequality  \ref{cond3}, we have the modified inequality is $q^{\frac{3m}{10}-1}> 84$. 

We already have that the condition is satisfied, when $q\geq 16$ and $m \geq 29$. Now from the above modified inequality we have the following conditions where the inequality is satisfied.

\begin{itemize}
\item When $q \geq 16$, condition holds for $m\geq 9$.
\item When $q=8$, condition holds for $m\geq 11$. 
\item When $q=4$, condition holds for $m \geq 14$.
\item When $q=2$, condition holds for $m\geq 25$.
\end{itemize}

For the remaining pairs we calculate the exact $\omega$ and take $r= q^m-1$, then all the pairs $(q,m)$ satisfy the condition except the following $(2, 6), (2, 8), (2, 9), (2, 10), (2, 11), (2, 12)$, $(2, 14), (2, 15), (2, 16), (2, 18), (2, 20), (2, 24), (4, 5), (4, 6), (4, 7), (4, 8), (4, 9), (4, 10), (4, 12)$, \\ $(8, 5),(8, 6), (8,  8), (8, 10), (16, 5), (16, 6), (16, 7), (32, 6), (64, 5), (64, 6), (128, 5), (256, 5).$ 
  
  Then by using appropriate values of $r$ in the modified prime sieving condition \ref{cond3} as shown in Table 2, we have the following possible exceptional pairs :
  $(2, 6), (2, 8), (2, 9), (2, 10), (2, 11)$, $(2, 12), (4, 5), (4, 6), (4, 7), (4, 8), (8, 5), (8, 6), (8, 8), (16, 5), (16, 6).$

\end{proof}

Following table illustrates some of the calculations to show the pairs $(q, m)$ satisfying the condition $q^{\frac{m}{2}-1}> \rit$.

\begin{tabular}{|c|c|c|c|c|c|c|}
\hline 
$(q,m)$ & prime factors of $q^m-1$ & $\omega (d)$ & $\Delta$ & $\Lambda$ & $q^{\frac{m}{2}-1}$ & $\rit$ \\ 
\hline 
(2, 14) & 3, 43, 127 & 1 & 0.93744 & 5.19918 & 64 & 41.5934 \\ 
\hline 
(2 ,15) & 7, 31, 151 & 1 & 0.922239 & 5.25295 & 90.5097 & 42.0236 \\ 
\hline 
(2, 16)& 3, 15, 17, 257 & 1 & 0.474571 & 12.5358 & 128 & 100.286  \\ 
\hline 
(2, 18) & 3, 7, 19, 73 & 1 & 0.581625 & 10.5966 & 256 & 84.7728 \\ 
\hline 
(2, 20) & 3, 5, 11, 31, 41 & 1 & 0.304885 & 24.9595 & 512 & 199.676 \\ 
\hline 
(2, 24) & 3, 5, 7, 13, 17, 241 & 2 & 0.434494 & 18.1107 & 2048 & 579.542 \\ 
\hline 
(4, 9) & 3, 7, 19, 73 & 1 & 0.581625 & 10.5966 & 128 & 84.7728 \\ 
\hline 
(4, 10) & 3, 5, 11, 31, 41& 1 & 0.304885 & 24.9595 & 256 & 199.676 \\ 
\hline 
(4, 12) & 3, 5, 7, 13, 17, 241 & 2 & 0.434494 & 18.1107 & 1024 & 579.542 \\ 
\hline 
(8, 10) & 3, 7, 11, 31, 151, 331 & 1 & 0.448664 & 22.0596 & 4096 & 176.477 \\ 
\hline 
(16, 7) & 3, 5, 29, 43, 113, 127 & 1 & 0.451076 & 21.9523 & 1024 & 175.618 \\ 
\hline 
(32, 6) & 3, 7, 11, 31, 151, 331 & 1 & 0.448664 & 22.0596 & 1024 & 176.477 \\ 
\hline
(64, 5) & 3, 7, 11, 31, 151, 331 & 1 & 0.448664 & 22.0596 & 512 & 176.477 \\ 
\hline 
(64, 6) & 3, 5, 7, 11, 13, 19, 37, 73, 109 & 2 & 0.355376 & 32.9531 & 4096 & 1054.52 \\ 
\hline
(128, 5) & 31, 71, 127, 122921 & 1 & 0.956067 & 7.22976 & 1448.5 & 57.8381 \\ 
 \hline
 (256, 5) & 3, 5, 11, 17, 19, 31, 41, 61681 & 2 & 0.587206 & 17.3268 & 4096 & 554.458 \\ 
\hline 
\end{tabular}

\begin{center}
Table 2
\end{center}

For Mersenne primes i.e., primes of the form $2^n-1$, we have the following theorem.
\begin{theorem}
If $2^m-1$ is a Mersenne prime such that $2^m-1> 13$, then the pairs $(2, m)\in \T$.
\end{theorem}
\textbf{Proof.}
If $2^m-1$ is a Mersenne prime i.e. $m= 3, 5, 7, 13, 17, 19$ etc., then every $\alpha \in \mathbb{F}^*_{2^m}$ is a primitive element other than 1. Also if $\alpha \in \mathbb{F}^*_{2^m}$, then degree of the minimal polynomial of $\alpha$ over $\mathbb{F}_2$ is greater than 3. Hence $a \alpha^2 + b\alpha +c \neq 0,1$. Thus $a \alpha^2 + b \alpha +c $ is also primitive. Moreover the trace map $Tr_{\mathbb{F}_{2^m}/\mathbb{F}_2}$ is onto and inverse image of every element in $\mathbb{F}_2$ contains $2^{m}-1$ elements in $\mathbb{F}_{2^m}$  and atleast three of them are primitive, hence the result. $\square$

\subsection{Algorithm for the programming}

\begin{theorem}
Let $\F$ be a finite field and $\Fm$ be the extension field of degree $m$. For $m \geq 5$ there always exists a primitive pair $(\alpha, f(\alpha))$   and $Tr_{\Fm/\F}(\alpha)=\beta$, for any prescribed $\beta\in\F$, where $f(x)= ax^2 + bx + c\in \Fm[x]$ such that $b^2-4ac\neq 0$; with the only exceptional pair $(q,m)$ which is $(2,6)$.
\end{theorem}

\begin{proof}
We examine the exceptional pairs of the above theorems by using the following algorithm and try to establish the exact exceptional pair (2, 6). For this case the field $\mathbb{F}_{2^6} \cong \frac{\mathbb{Z}_2[x]}{<y(x)>}$, where $y(x)= x^6+ x^4 + x^3 + x+ 1$. For this field the quadratic polynomial is $f(x) = \alpha x^2 + (\alpha^5+ \alpha^4 + \alpha +1 )$.  
\end{proof}

\begin{conjecture}
For the fields $\Fm$, where $m= 3, 4$, by sufficient amount of calculation using the algorithm, we have the pairs (3, 3), (5, 3), (6, 3), (3, 4) as the only possible exceptional pairs. 
\end{conjecture}

Among the pairs (3, 3) is an exceptional pair, where $\mathbf{F_{3^3}} \cong \frac{\mathbb{Z}_3[x]}{<x^3+2x+1>}$ and the quadratic polynomial is $f(x) = \alpha x^2 + (\alpha +1) x+ (2 \alpha^2+ 2)$  ($\alpha$ is a primitive root of $\mathbb{F}_{3^3}$).

We use SAGE math for the computation by using the following algorithm.

\noindent\rule{\textwidth}{0.2pt}
\textbf{Algorithm} Check if $(q, m)$ is exceptional pair or not\\
\noindent\rule{\textwidth}{0.2pt}

\begin{align*}
Input: \Fm \quad  &=  \quad GF(q, m)\\
  def \quad is\_primitive \quad &(b, q, m, \Fm):\\  
   \mathbf{if} & \quad b = = 0:\\
   & \mathbf{return} \quad false\\
    \mathbf{if} & \quad multiplicative\_order(b) = = q^m-1:\\
   & \mathbf{return} \quad true\\
   & \mathbf{else} \quad next \quad b\\
   def \quad prescribed\_trace \quad & (b, n, m, \Fm):\\
   Tr(b) \quad & =\,  Sum \left( \left[ b\wedge (q\wedge i)\, for \, i \, in \, \left[0 \dots m-1\right] \right] \right)\\
   T \, & = \, store \, Tr(b) \, for \, all \, b \in \Fm\\
    \mathbf{if} & \, T = GF(q, 1)\\
    & \mathbf{return} \, true \, and \, proceed \\
    \mathbf{else}\,  & \, next \, b\\
  def\, check\_coeffs\,  & \, (coefs, q, m, \Fm,  alpha):\\  
   a \, & = \, coefs[0]\\
   b \, & = \, coefs[1]\\
   c \, & = \, coefs[2]\\
   \mathbf{if} \, & b\wedge2 \, = = \, 4\ast a \ast c:\\
   & \mathbf{return} \, true\\
   for \, & i \, in \, xrange\, (1, q\wedge m):\\   
   & \mathbf{if} \, gcd (i, q\wedge m -1) \, = = \, 1\\   
   & \mathbf{if} \, is\_primitive \, \left( a\ast alpha\wedge (2 \ast i) \, + \, b\ast alpha \wedge i \, + \, c, \, q, m, \Fm   \right):\\
   & \mathbf{reuturn} \, true\\
   def\, check\_pair \, & \, (q, m):\\ 
   \Fm.<alpha> \, & \, = \, GF(q\wedge m, \, modulus \, = \, ``primitive"):\\
   \mathbf{print} \, & \, ``The \, Modulus \, is", \, \Fm.modulus():\\ 
   \mathbf{for} \, C1\, & in \, \, xrange(1,\, q\wedge m):\\
   \mathbf{for} \, & \, C2 \, in \, xrange(0, \, q\wedge m):\\
   & \mathbf{for} \, C3 \, in \, xrange(0, \, q\wedge m):\\
   & coefs \, = \, \left[ list(\Fm) [C1], \, list(\Fm)[C2],\, list(\Fm)[C3]\right]\\ 
   & \mathbf{if}  \, check\_coefs\, \left( coefs, \, q, \, m, \, \Fm, \, alpha \right)\, == \, false\\
   & \mathbf{print} \, coefs\\
   & \mathbf{return} \, false 
   \end{align*}

\end{document}